\documentclass[12pt]{article}
\usepackage{amsmath}
\usepackage{amsthm}
\usepackage{amsfonts}
\usepackage{eucal}
\newtheorem{thm}{Theorem}
\newtheorem{prop}{Proposition}
\newtheorem{lem}{Lemma}

\newtheorem{cor}{Corollary}
\def\zt{\zeta}
\def\si{\sigma}
\def\Zt{\mathfrak Z}
\def\D{\mathcal D}
\def\Sym{\operatorname{Sym}}
\def\Iso{\operatorname{Iso}}
\def\Q{\mathbf Q}
\def\R{\mathbf R}
\def\F{\mathcal F}
\newcommand{\pD}[2]{\frac{\partial{#1}}{\partial{#2}}}
\begin{document}
\begin{titlepage}
\begin{flushleft} DESY 12-091 \hfill {\tt arXiv:1205.7051 [math.NT]}
\end{flushleft}
\vspace{0.8 cm}
\begin{center}
\Large
{\bf On Multiple Zeta Values of Even Arguments}\\
\vspace{1.0 cm}
\large
{Michael E. Hoffman\footnote{Supported by a grant from the
German Academic Exchange Service (DAAD) during the preparation
of this paper.  The author also thanks DESY for providing 
facilities and financial support for travel, and the referee
for urging him to include multiple zeta-star values (which
led to Corollary \ref{Fst}).}\\
\vspace{0.5 cm}
\normalsize
{\it Dept. of Mathematics, U. S. Naval Academy\\
Annapolis, MD 21402 USA\\
and\\
Deutsches Elektronen-Synchrotron DESY\\
Platanenalle 6, D-15738 Zeuthen, Germany}\\
\vspace{0.5 cm}
{\tt meh@usna.edu}\\
\vspace{0.5 cm}
March 17, 2016
\small MR Classifications: Primary 11M32; Secondary 05E05, 11B68}
\vfill
\end{center}
\begin{abstract}
For $k\le n$, let $E(2n,k)$ be the sum of all multiple zeta
values with even arguments whose weight is $2n$ and whose
depth is $k$.  Of course $E(2n,1)$ is the value $\zt(2n)$
of the Riemann zeta function at $2n$, and it is well known
that $E(2n,2)=\frac34\zt(2n)$.  Recently Z. Shen and T. Cai 
gave formulas for $E(2n,3)$ and $E(2n,4)$ in terms of $\zt(2n)$
and $\zt(2)\zt(2n-2)$.  We give two formulas for $E(2n,k)$,
both valid for arbitrary $k\le n$, one of which generalizes 
the Shen-Cai results; by comparing the two we obtain a 
Bernoulli-number identity.  We also give explicit generating 
functions for the numbers $E(2n,k)$ and for the analogous
numbers $E^{\star}(2n,k)$ defined using multiple zeta-star
values of even arguments.
\end{abstract}
\vfill
\end{titlepage}
\section{Introduction and Statement of Results}
For positive integers $i_1,\dots,i_k$ with $i_1>1$, we define
the multiple zeta value $\zt(i_1,\dots,i_k)$ by
\begin{equation}
\label{mzv}
\zt(i_1,\dots,i_k)=\sum_{n_1>\dots>n_k\ge 1}\frac1{n_1^{i_1}\cdots n_k^{i_k}} .
\end{equation}
The multiple zeta value (\ref{mzv}) is said to have weight $i_1+\dots+i_k$
and depth $k$.
Many remarkable identities have been proved about these numbers,
but in this note we will concentrate on the case
where the $i_j$ are even integers.  Let $E(2n,k)$ be the
sum of all the multiple zeta values of even-integer arguments
having weight $2n$ and depth $k$, i.e.,
\[
E(2n,k)=\sum_{\substack{\text{$i_1,\dots, i_k$ even}\\ i_1+\dots+i_k=2n}}\zt(i_1,\dots,i_k) .
\]
Of course 
\begin{equation}
\label{euler}
E(2n,1)=\zt(2n)=\frac{(-1)^{n-1}B_{2n}(2\pi)^{2n}}{2(2n)!},
\end{equation}
where $B_{2n}$ is the $2n$th Bernoulli number,
by the classical formula of Euler.  Euler also studied double zeta
values (i.e., multiple zeta values of depth 2) and in his paper \cite{E}
gave two identities which read
\begin{align*}
\sum_{i=2}^{2n-1}(-1)^i\zt(i,2n-i)&=\frac12\zt(2n) \\
\sum_{i=2}^{2n-1}\zt(i,2n-i)&=\zt(2n)
\end{align*}
in modern notation.  From these it follows that
\[
E(2n,2)=\frac34\zt(2n),
\]
though Gangl, Kaneko and Zagier \cite{GKZ} seem to be the first
to have pointed it out in print.  
Recently Shen and Cai \cite{SC} proved the formulas
\begin{align}
\label{trip}
E(2n,3)&=\frac58\zt(2n)-\frac14\zt(2)\zt(2n-2),\ n\ge 3 \\
\label{quad}
E(2n,4)&=\frac{35}{64}\zt(2n)-\frac5{16}\zt(2)\zt(2n-2),\ n\ge 4 .
\end{align}
Identity (\ref{trip}) was also proved by Machide \cite{M}
using a different method.
\par
This begs the question whether there is a general formula of
this type for $E(2n,k)$.  
The pattern
\[
\frac34,\quad \frac34\cdot\frac56=\frac58,\quad
\frac34\cdot\frac56\cdot\frac78=\frac{35}{64}
\]
of the leading coefficients makes one curious.
In fact, the general result is as follows.
\begin{thm}
\label{ztcon}
For $k\le n$, 
\begin{multline*}
E(2n,k)=\frac1{2^{2(k-1)}}\binom{2k-1}{k}\zt(2n)\\
-\sum_{j=1}^{\lfloor \frac{k-1}2\rfloor}
\frac1{2^{2k-3}(2j+1)B_{2j}}\binom{2k-2j-1}{k}\zt(2j)\zt(2n-2j) .
\end{multline*}
\end{thm}
The next two cases after (\ref{quad}) are
\begin{align*}
E(2n,5) &= \frac{63}{128}\zt(2n)-\frac{21}{64}\zt(2)\zt(2n-2)
+\frac{3}{64}\zt(4)\zt(2n-4)\\
E(2n,6) &= \frac{231}{512}\zt(2n)-\frac{21}{64}\zt(2)\zt(2n-2)
+\frac{21}{256}\zt(4)\zt(2n-4).
\end{align*}
\par
We prove Theorem \ref{ztcon} in \S3 below, using the generating function
\[
F(t,s)=1+\sum_{n\ge k\ge 1} E(2n,k)t^ns^k .
\]
In \S2 we establish the following explicit formula.
\begin{thm}
\label{gfun}
\[
F(t,s)=\frac{\sin(\pi\sqrt{1-s}\sqrt{t})}{\sqrt{1-s}\sin(\pi\sqrt{t})} .
\]
\end{thm}
Our proof uses symmetric functions.
We define a homomorphism $\Zt:\Sym\to\R$, where $\Sym$ is the algebra of 
symmetric functions, and a family $N_{n,k}\in\Sym$ such that
$\Zt$ sends $N_{n,k}$ to $E(2n,k)$.  We then obtain a formula
for the generating functions
\[
\F(t,s)=1+\sum_{n\ge k\ge 1} N_{n,k}t^n s^k\in \Sym[[t,s]]
\]
and apply $\Zt$ to get Theorem \ref{gfun}.
\par
Along with the multiple zeta values there are the multiple zeta-star values
\[
\zt^{\star}(i_1,\dots,i_k)=\sum_{n_1\ge\dots\ge n_k\ge 1}\frac1{n_1^{i_1}\cdots n_k^{i_k}},
\]
where the strict inequalities in equation (\ref{mzv}) are replaced by
$\ge$.
These coincide with multiple zeta values for depth 1, and for greater
depths they are simply sums of multiple zeta values, e.g.,
\[
\zt^{\star}(6,2,4)=\zt(6,2,4)+\zt(8,4)+\zt(6,6)+\zt(12) .
\]
Conversely, multiple zeta values are sums of multiple zeta-star values
with signs alternating by depth, e.g.,
\[
\zt(6,2,4)=\zt^{\star}(6,2,4)-\zt^{\star}(8,4)-\zt^{\star}(6,6)+\zt^{\star}(12) .
\]
For $k\le n$ we can define $E^{\star}(2n,k)$ as the sum of all multiple
zeta-star values with even arguments having weight $2n$ and depth $k$.
The generating function
\[
F^{\star}(t,s)=1+\sum_{n\ge k\ge 1}E^{\star}(2n,k)t^ns^k
\]
turns out to have a remarkably simple relation to $F(t,s)$, as we
show in \S2.
\begin{cor}
\label{Fst}
The generating functions $F(t,s)$ and $F^{\star}(t,s)$ are related by
$F^{\star}(t,s)=F(t,-s)^{-1}$.
\end{cor}
\par
From the form of $\F(t,s)$ we show that it satisfies
a partial differential equation (Proposition \ref{pde} below),
which is equivalent to a recurrence for the $N_{n,k}$.  
From the latter we obtain a formula for the $N_{n,k}$ in terms of 
complete and elementary symmetric functions, to which $\Zt$
can be applied to give the following alternative formula for $E(2n,k)$.
\begin{thm} For $k\le n$,
\label{explicit}
\[
E(2n,k)=\frac{(-1)^{n-k-1}\pi^{2n}}{(2n+1)!}\sum_{i=0}^{n-k}
\binom{n-i}{k}\binom{2n+1}{2i}2(2^{2i-1}-1)B_{2i} .
\]
\end{thm}
Note that the sum given by Theorem \ref{explicit} has $n-k+1$ terms,
while that given by Theorem \ref{ztcon} has $\lfloor\frac{k-1}2\rfloor+1$
terms.
Yet another explicit formula for $E(2n,k)$ can be obtained by setting
$d=1$ in Theorem 7.1 of Komori, Matsumoto and Tsumura \cite{KMT}.
That formula expresses $E(2n,k)$ as a sum over partitions of $k$,
and it is not immediately clear how it relates to our two formulas.
\par
Comparison of Theorems \ref{ztcon} and \ref{explicit} establishes
the following Bernoulli-number identity.
\begin{thm}
\label{bercon}
For $k\le n$,
\begin{multline*}
\sum_{i=0}^{\lfloor \frac{k-1}2 \rfloor}\binom{2k-2i-1}{k}
\binom{2n+1}{2i+1}B_{2n-2i} =\\
(-1)^k2^{2k-2n}
\sum_{i=0}^{n-k}\binom{n-i}{k}\binom{2n+1}{2i}(2^{2i-1}-1)B_{2i} .
\end{multline*}
\end{thm}
\par\noindent
It is interesting to contrast this result with the Gessel-Viennot
identity (see \cite[Theorem 4.2]{CS}) valid on the complementary
range:
\begin{equation}
\label{GV}
\sum_{i=0}^{\lfloor \frac{k-1}2 \rfloor}\binom{2k-2i-1}{k}
\binom{2n+1}{2i+1}B_{2n-2i}=
\frac{2n+1}{2}
\binom{2k-2n}{k} ,\quad k > n .
\end{equation}
Note that the right-hand side of equation (\ref{GV}) is zero unless $k\ge 2n$.
\section{Symmetric Functions}
We think of $\Sym$ as the subring of $\Q[[x_1,x_2,\dots]]$ consisting
of those formal power series of bounded degree that are invariant
under permutations of the $x_i$.
A useful reference is the first chapter of Macdonald \cite{Mac}.
We denote the elementary, complete,
and power-sum symmetric functions of degree $i$ by $e_i$, $h_i$, and $p_i$
respectively.  They have associated generating functions
\begin{align*}
E(t)&=\sum_{j=0}^\infty e_jt^j=\prod_{i=1}^\infty (1+tx_i)\\
H(t)&=\sum_{j=0}^\infty h_jt^j=\prod_{i=1}^\infty \frac1{1-tx_i} = E(-t)^{-1}\\
P(t)&=\sum_{j=1}^\infty p_jt^{j-1}=\sum_{i=1}^\infty \frac{x_i}{1-tx_i}=
\frac{H'(t)}{H(t)} .
\end{align*}
\par
As explained in \cite{H2} and in greater detail in \cite{H3}, 
there is a homomorphism
$\zt:\Sym^0\to\R$, where $\Sym^0$ is the subalgebra of $\Sym$ generated
by $p_2,p_3,p_4,\dots$, such that $\zt(p_i)$ is the value $\zt(i)$ of
the Riemann zeta function at $i$, for $i\ge 2$ (in \cite{H2,H3} this
homomorphism is extended to all of $\Sym$, but we do not need the
extension here).  
Let $\D:\Sym\to\Sym$ be the degree-doubling map that sends $x_i$ to 
$x_i^2$.  Then $\D(\Sym)\subset\Sym^0$, so the composition 
$\Zt=\zt\D$ is defined on all of $\Sym$.  
(Alternatively, we can simply think of $\Zt$ as sending 
$x_i$ to $1/i^2$:  see \cite[Ch. I, \S2, ex. 21]{Mac}.)
Note that $\Zt(p_i)=\zt(2i)\in\R$.  
Further, $\Zt$ sends the monomial
symmetric function $m_{i_1,\dots,i_k}$ to the symmetrized sum
of multiple zeta values
\[
\frac1{|\Iso (i_1,\dots,i_k)|}
\sum_{\si\in S_k}\zt(2i_{\si(1)},2i_{\si(2)},\dots,2i_{\si(k)}),
\]
where $S_k$ is the symmetric group on $k$ letters and 
$\Iso (i_1,\dots,i_k)$ is the subgroup of $S_k$ that fixes
$(i_1,\dots,i_k)$ under the obvious action.
\par
Now let
$N_{n,k}$ be the sum of all the monomial symmetric functions 
corresponding to partitions of $n$ having length $k$.
Of course $N_{n,k}=0$ unless $k\le n$, and $N_{k,k}=e_k$.
Then $\Zt$ sends $N_{n,k}$ to $E(2n,k)$.
Also, if we define (as in the introduction)
\[
\F(t,s)=1+\sum_{n\ge k\ge 1}N_{n,k}t^ns^k ,
\]
then $\Zt$ sends $\F(t,s)$ to the generating function $F(t,s)$.
We have the following simple description of $\F(t,s)$.
\begin{lem}
\label{infprod}
$\F(t,s)=E((s-1)t)H(t)$.
\end{lem}
\begin{proof}
Evidently $\F(t,s)$ has the formal factorization
\[
\prod_{i=1}^{\infty}(1+stx_i+st^2x_i^2+\cdots)
=\prod_{i=1}^{\infty}\frac{1+(s-1)tx_i}{1-tx_i}
=E((s-1)t)H(t) .
\]
\end{proof}
\begin{proof}[Proof of Theorem \ref{gfun}]
Using the well-known formula for $\zt(2,2,\dots,2)$ \cite[Cor. 2.3]{H1},
\begin{equation}
\label{base}
\Zt(e_n)=\zt(\underbrace{2,2,\dots,2}_n)=\frac{\pi^{2n}}{(2n+1)!} .
\end{equation}
Hence
\[
\Zt(E(t))=\frac{\sinh(\pi\sqrt t)}{\pi\sqrt t} ,
\]
and the image of $H(t)=E(-t)^{-1}$ is
\[
\Zt(H(t))=
\frac{\pi\sqrt{-t}}{\sinh(\pi\sqrt{-t})}=
\frac{\pi\sqrt{t}}{\sin(\pi\sqrt{t})} .
\]
Thus from Lemma \ref{infprod} $F(t,s)=\Zt(\F(t,s))$ is
\[
\Zt(E((s-1)t)H(t))=\frac{\sinh(\pi\sqrt{(s-1)t})}{\pi\sqrt{(s-1)t}}
\frac{\pi\sqrt{t}}{\sin(\pi\sqrt t)}=
\frac{\sin(\pi\sqrt{(1-s)t})}{\sqrt{1-s}\sin(\pi\sqrt t)} .
\]
\end{proof}
Taking limits as $s\to 1$ in Theorem \ref{gfun}, we obtain
\[
F(t,1)=\frac{\pi\sqrt t}{\sin\pi\sqrt t}
\]
and so, taking the coefficient of $t^n$, the following result.
\begin{cor}
For all $n\ge 1$,
\label{edge}
\[
\sum_{k=1}^n E(2n,k)
=\frac{2(2^{2n-1}-1)(-1)^{n-1}B_{2n}\pi^{2n}}{(2n)!}=2(1-2^{1-2n})\zt(2n) .
\]
\end{cor}
\par\noindent
{\bf Remark.} This result was obtained previously by Aoki, Kombu and 
Ohno \cite{AKO}, who stated it in the language of multiple zeta-star
values; since
\begin{equation}
\label{st2}
\zt^{\star}(\underbrace{2,\dots,2}_n)=E(2n,n)+E(2n,n-1)+\dots+E(2n,1),
\end{equation}
Corollary \ref{edge} can be recognized as \cite[equation (4.6)]{AKO}.
In fact, using Euler's infinite product for sine, one sees that
\[
F(z^2,1)=\prod_{m=1}^\infty\left(1-\frac{z^2}{m^2}\right)^{-1} 
=1+\sum_{n\ge 1}\zt^{\star}(\underbrace{2,\dots,2}_n)z^{2n} .
\]
\par
Recall that we defined $E^{\star}(2n,k)$ as the sum of all multiple
zeta-star values with even arguments having depth $k$ and weight $2n$.
The left-hand side of equation (\ref{st2}) is $E^{\star}(2n,n)$, and
that equation generalizes as follows.
\begin{lem}
\label{Est}
For $n\ge k\ge 1$,
\[
E^{\star}(2n,k)=\sum_{j=1}^k \binom{n-j}{k-j}E(2n,j) .
\]
\end{lem}
\begin{proof}
Let $I=(i_1,\dots,i_k)$ be a composition (i.e., ordered partition) of $n$.  
We can think of $I$ as specified by placing $k-1$ dividers within a row 
of $n$ dots, which makes it clear that there are $\binom{n-1}{k-1}$ 
compositions of $n$ with $k$ parts.
If we associate to $I$ the multiple zeta and multiple star zeta values
\[
\zt(2I)=\zt(2i_1,\dots,2i_k),\quad
\zt^{\star}(2I)=\zt^{\star}(2i_1,\dots,2i_k),
\]
of even values, then $\zt^{\star}(2I)=\sum_{I\succeq J}\zt(2J)$, where
$\succeq$ is the partial order on compositions given by refinement,
i.e., $I\succeq J$ if $J$ can be obtained by combining adjacent parts of $I$;
in terms of the dividers-in-row-of-dots picture, $J$ is obtained by 
removing some dividers from $I$.
\par
Now $E^{\star}(2n,k)$ is the sum of all $\zt^{\star}(2I)$ with $I$ having
$k$ parts.  Write each of these as a sum of multiple zeta values.  Then
the coefficient of $\zt(2J)$, where $J$ has $j\le k$ parts, is the number
of distinct compositions $I$ with $k$ parts such that $J\preceq I$;
this corresponds to the number of ways to insert $k-j$ additional
dividers into $J$.  Since there are $n-1-(j-1)=n-j$ places to put them,
this number is $\binom{n-j}{k-j}$.
\end{proof}
\begin{proof}[Proof of Corollary \ref{Fst}]
Using Lemma \ref{Est} and Theorem \ref{gfun},
\begin{multline*}
F^{\star}(t,s)=1+\sum_{n\ge k\ge 1}E^{\star}(2n,k)t^ns^k
=1+\sum_{n\ge k\ge 1}\sum_{j=1}^k\binom{n-j}{k-j}E(2n,j)t^ns^k\\
=1+\sum_{n\ge j\ge 1}E(2n,j)\sum_{i=0}^{n-j}\binom{n-j}{i}t^ns^{j+i}
=1+\sum_{n\ge j\ge 1}E(2n,j)t^ns^j(1+s)^{n-j}\\
=F\left(t(1+s),\frac{s}{1+s}\right)
=\frac{\sin\left(\pi\sqrt{1-\frac{s}{1+s}}\sqrt{t(1+s)}\right)}
{\sqrt{1-\frac{s}{1+s}}\sin(\pi\sqrt{t(1+s)})}
=\frac{\sqrt{1+s}\sin(\pi\sqrt{t})}{\sin(\pi\sqrt{(1+s)t})}\\
=\frac1{F(t,-s)} .
\end{multline*}
\end{proof}
\par
Another consequence of Lemma \ref{infprod} is the following
partial differential equation.
\begin{prop}
\label{pde}
\[
t\pD{\F}{t}(t,s)+(1-s)\pD{\F}{s}(t,s)=tP(t)\F(t,s) .
\]
\end{prop}
\begin{proof}
From Lemma \ref{infprod} we have
\begin{align*}
\pD{\F}{t}(t,s)&=(s-1)E'((s-1)t)H(t)+E((s-1)t)H'(t)\\
\pD{\F}{s}(t,s)&=tE'((s-1)t)H(t)
\end{align*}
from which the conclusion follows.
\end{proof}
\par
Now examine the coefficient of $t^ns^k$ in Proposition
\ref{pde} to get the following.
\begin{prop} For $n\ge k+1$,
\label{sfi}
\[
p_1N_{n-1,k}+p_2N_{n-2,k}+\dots+p_{n-k}N_{k,k}=(n-k)N_{n,k}+(k+1)N_{n,k+1} .
\]
\end{prop}
\par\noindent
It is also possible to prove this result directly via a 
counting argument like that used to prove the lemma of
\cite[p. 16]{H3}.
\par
The preceding result allows us to write $N_{n,k}$
explicitly in terms of complete and elementary symmetric 
functions as follows.
\begin{lem}
\label{Nexp}
For $r\ge 0$,
\[
N_{k+r,k}=\sum_{i=0}^r (-1)^i\binom{k+i}{i}h_{r-i}e_{k+i} .
\]
\end{lem}
\begin{proof}
We use induction on $r$, the result being evident for $r=0$.
Proposition \ref{sfi} gives
\[
\sum_{i=1}^{r+1}p_iN_{k+r+1-i,k}=(r+1)N_{k+r+1,k}+(k+1)N_{k+r+1,k+1} ,
\]
which after application of the induction hypothesis becomes
\begin{multline*}
\sum_{i=1}^{r+1}\sum_{j=0}^{r+1-j}(-1)^jp_i\binom{k+j}{j}h_{r+1-i-j}N_{k+j,k+j}=\\
(r+1)N_{k+r+1,k}+(k+1)\sum_{j=0}^r\binom{k+1+j}{j}h_{r-j}N_{k+1+j,k+1+j} .
\end{multline*}
The latter equation can be rewritten
\begin{multline*}
\sum_{j=0}^r(-1)^j\binom{k+j}{j}N_{k+j,k+j}\sum_{i=1}^{r+1-j}p_ih_{r+1-i-j}=\\
(r+1)N_{k+r+1,k}-(k+1)\sum_{j=1}^{r+1}(-1)^j\binom{k+j}{j-1}h_{r+1-j}N_{k+j,k+j} .
\end{multline*}
Now the inner sum on the left-hand side is $(r+1-j)h_{r+1-j}$ by
the recurrence relating the complete and power-sum symmetric functions,
so we have
\begin{multline*}
(r+1)N_{k+r+1,k} - (r+1)N_{k,k}h_{r+1}=\\
\sum_{j=1}^{r+1}(-1)^jh_{r+1-j}N_{k+j,k+j}
\left((r+1-j)\binom{k+j}{j}+(k+1)\binom{k+j}{j-1}\right) ,
\end{multline*}
and the conclusion follows after the observation that
$(k+1)\binom{k+j}{j-1}=j\binom{k+j}{j}$.
\end{proof}
\begin{proof}[Proof of Theorem \ref{explicit}]
Rewrite Lemma \ref{Nexp} in the form
\[
N_{n,k}=\sum_{i=0}^{n-k}\binom{n-i}{k}(-1)^{n-k-i}h_ie_{n-i} 
\]
and apply the homomorphism $\Zt$, using equation (\ref{base}) and
\[
\Zt(h_i)=\frac{2(2^{2i-1}-1)(-1)^{i-1}B_{2i}\pi^{2i}}{(2i)!} .
\]
\end{proof}
\section{Proof of Theorems \ref{ztcon} and \ref{bercon}}
From the introduction we recall the statement of Theorem \ref{ztcon}:
\begin{multline*}
E(2n,k)=\frac1{2^{2(k-1)}}\binom{2k-1}{k}\zt(2n)\\
-\sum_{j=1}^{\lfloor \frac{k-1}2\rfloor}
\frac1{2^{2k-3}(2j+1)B_{2j}}\binom{2k-2j-1}{k}\zt(2j)\zt(2n-2j) .
\end{multline*}
We note that Euler's formula (\ref{euler})
can be used to write the result in the alternative form
\begin{equation}
\label{altcon}
E(2n,k)=\sum_{j=0}^{\lfloor\frac{k-1}2\rfloor}
\frac{(-1)^j\pi^{2j}\zt(2n-2j)}{2^{2k-2j-2}(2j+1)!}\binom{2k-2j-1}{k}
\end{equation}
which avoids mention of Bernoulli numbers.
\par
We now expand out the generating function $F(t,s)$.  We have
\begin{multline*}
F(t,s)=\frac1{\sqrt{1-s}\sin\pi\sqrt t}\sin(\pi\sqrt t\sqrt{1-s})\\
=\frac{\pi\sqrt t}{\sin\pi\sqrt t}\sum_{j=0}^\infty
\frac{(-1)^j\pi^{2j}t^j(1-s)^j}{(2j+1)!}
=\sum_{k=0}^\infty s^kG_k(t),
\end{multline*}
where
\begin{equation}
\label{Gdef}
G_k(t)=(-1)^k\frac{\pi\sqrt t}{\sin\pi\sqrt t}
\sum_{j\ge k}\frac{(-1)^j\pi^{2j}t^j}{(2j+1)!}\binom{j}{k} .
\end{equation}
Then Theorem \ref{ztcon} is equivalent to the statement that
\[
G_k(t)=\sum_{n\ge k}t^n\sum_{j=0}^{\lfloor\frac{k-1}2\rfloor}
\frac{(-1)^j\pi^{2j}\zt(2n-2j)}{2^{2k-2j-2}(2j+1)!}\binom{2k-2j-1}{k} 
\]
for all $k$.
We can write the latter sum as
\begin{multline}
\label{hard}
\sum_{j=0}^{\lfloor\frac{k-1}2\rfloor}
\frac{(-4\pi^2t)^j}{2^{2k-2}(2j+1)!}\binom{2k-2j-1}{k}
\sum_{n\ge j+1}\zt(2n-2j)t^{n-j}-\\
\sum_{j=0}^{\lfloor\frac{k-1}2\rfloor}
\frac{(-4\pi^2t)^j}{2^{2k-2}(2j+1)!}\binom{2k-2j-1}{k}
\sum_{n=j+1}^{k-1}\zt(2n-2j)t^{n-j}=\\
\frac12(1-\pi\sqrt t\cot\pi\sqrt t)
\sum_{j=0}^{\lfloor\frac{k-1}2\rfloor}
\frac{(-4\pi^2t)^j}{2^{2k-2}(2j+1)!}\binom{2k-2j-1}{k}-\\
\sum_{j=0}^{\lfloor\frac{k-1}2\rfloor}
\frac{(-4\pi^2t)^j}{2^{2k-2}(2j+1)!}\binom{2k-2j-1}{k}
\sum_{n=j+1}^{k-1}\zt(2n-2j)t^{n-j} ,
\end{multline}
where we have used the generating function
\[
\frac12(1-\pi\sqrt t\cot \pi\sqrt t)=\sum_{i=1}^\infty \zt(2i)t^i .
\]
Note that the last sum in (\ref{hard}) is a polynomial that cancels
exactly those terms in 
\begin{equation}
\label{expr}
\frac12(1-\pi\sqrt t\cot\pi\sqrt t)
\sum_{j=0}^{\lfloor\frac{k-1}2\rfloor}
\frac{(-4\pi^2t)^j}{2^{2k-2}(2j+1)!}\binom{2k-2j-1}{k}
\end{equation}
of degree less than $k$.  Thus, to prove Theorem \ref{ztcon} it
suffices to show that
\[
G_k(t)=\text{terms of degree $\ge k$ in expression (\ref{expr}).}
\]
\par
From equation (\ref{Gdef}) it is evident that
\begin{equation}
\label{differ}
G_k(t)=\frac{\pi\sqrt t}{\sin\pi\sqrt t}\cdot\frac{(-t)^k}{k!}\cdot
\frac{d^k}{dt^k}\left(\frac{\sin\pi\sqrt t}{\pi\sqrt t}\right).
\end{equation}
We use this to obtain an explicit formula for $G_k(t)$.
\begin{lem}
\label{Gfun}
For $k\ge 0$,
\[
G_k(t)=P_k(\pi^2t)\pi\sqrt t\cot\pi\sqrt t+Q_k(\pi^2t),
\]
where $P_k,Q_k$ are polynomials defined by
\begin{align*}
P_k(x)=&-\sum_{j=0}^{\lfloor\frac{k-1}2\rfloor}\frac{(-4x)^j}{2^{2k-1}(2j+1)!}
\binom{2k-2j-1}{k} \\
Q_k(x)=&\sum_{j=0}^{\lfloor\frac{k}2\rfloor}\frac{(-4x)^j}{2^{2k}(2j)!} 
\binom{2k-2j}{k} .
\end{align*}
\end{lem}
\begin{proof}
In view of equation (\ref{differ}), the conclusion is equivalent to
\[
f^{(k)}(t)=(-1)^kk!t^{-k}P_k(\pi^2t)\cos\pi\sqrt t + (-1)^kk!t^{-k}Q_k(\pi^2t)
f(t),
\]
where $f(t)=\sin\pi\sqrt t/\pi\sqrt t$.
Differentiating, one sees that the polynomials $P_k$ and 
$Q_k$ are determined by the recurrence
\begin{align*}
(k+1)P_{k+1}(x)&=kP_k(x)-xP_k'(x)-\frac12Q_k(x)\\
(k+1)Q_{k+1}(x)&=\frac{2k+1}2Q_k(x)-xQ_k'(x)+\frac{x}2P_k(x)
\end{align*}
together with the initial conditions $P_0(x)=0$, $Q_0(x)=1$.
The recurrence and initial conditions are satisfied by the explicit
formulas above.
\end{proof}
\begin{proof}[Proof of Theorem \ref{ztcon}]
Using Lemma \ref{Gfun}, we have
\begin{multline*}
G_k(t)=-\sum_{j=0}^{\lfloor \frac{k-1}2\rfloor}\frac{(-4\pi^2t)^j}
{2^{2k-1}(2j+1)!}\binom{2k-2j-1}{k}\pi\sqrt t\cot\pi\sqrt t\\
+\sum_{j=0}^{\lfloor \frac{k}2\rfloor}\frac{(-4\pi^2t)^j}{2^{2k}(2j)!}
\binom{2k-2j}{k}=\\
\frac12(1-\pi\sqrt t\cot\pi\sqrt t)
\sum_{j=0}^{\lfloor \frac{k-1}2\rfloor}\frac{(-4\pi^2t)^j}
{2^{2k-2}(2j+1)!}\binom{2k-2j-1}{k}\\
+\text{terms of degree $< k$} ,
\end{multline*}
and this completes the proof.
\end{proof}
\begin{proof}[Proof of Theorem \ref{explicit}]
Using Theorem \ref{ztcon} in the form of equation (\ref{altcon}), eliminate
$\zt(2n-2j)$ using Euler's formula (\ref{euler}) and then
compare with Theorem \ref{explicit} to get
\begin{multline*}
\sum_{j=0}^{\lfloor \frac{k-1}2\rfloor}\frac{(-1)^{n-1}\pi^{2n}B_{2n-2j}}
{2^{2k-2n-1}(2n-2j)!(2j+1)!}\binom{2k-2j-1}{k}=\\
\frac{(-1)^{n-k-1}\pi^{2n}}{(2n+1)!}\sum_{i=0}^{n-k}
\binom{n-i}{k}\binom{2n+1}{2i}2(2^{2i-1}-1)B_{2i} .
\end{multline*}
Now multiply both sides by $(-1)^{n-1}2^{2k-2n-1}\pi^{-2n}(2n+1)!$
and rewrite the factorials on the left-hand side as a binomial
coefficient.
\end{proof}


\begin{thebibliography}{99}
\bibitem{AKO}
T. Aoki, Y. Kombu, and Y. Ohno, A generating function for sums of
multiple zeta values and its applications, \emph{Proc. Amer. Math. Soc.}
{\bf 136} (2008), 387-395.
\bibitem{CS}
W. Y. C. Chen and L. H. Sun, Extended Zeilberger's algorithm for
identities on Bernoulli and Euler polynomials, \emph{J. Number Theory}
{\bf 129} (2009), 2111-2132.
\bibitem{E}
L. Euler, Meditationes circa singulare serierum genus, \emph{Novi
Comm. Acad. Sci. Petropol.} {\bf 20} (1776), 140-186; reprinted
in \emph{Opera Omnia}, ser. I, vol. 15, B. G. Teubner, Berlin,
1927, pp. 217-267.
\bibitem{GKZ}
H. Gangl, M. Kaneko, and D. Zagier, Double zeta values and modular
forms, in \emph{Automorphic Forms and Zeta Functions}, S. B\"ocherer
{\it et. al.} (eds.), World Scientific, Singapore, 2006, pp. 71-106.
\bibitem{H1}
M. E. Hoffman, Multiple harmonic series, \emph{Pacific J. Math.}
{\bf 152} (1992), 275-290.
\bibitem{H2}
M. E. Hoffman, The algebra of multiple harmonic series, \emph{J. Algebra}
{\bf 194} (1997), 477-495.
\bibitem{H3}
M. E. Hoffman, A character on the quasi-symmetric functions coming
from multiple zeta values, \emph{Electron. J. Combin.} {\bf 15} (2008),
res. art. 97.
\bibitem{KMT} 
Y. Komori, K. Matsumoto and H. Tsumura, A study on multiple zeta values
from the viewpoint of zeta-functions of root systems, \emph{Funct.
Approx. Comment. Math.} {\bf 51} (2014), 43-76.
\bibitem{Mac}
I. G. Macdonald, \emph{Symmetric Functions and Hall Polynomials},
2nd ed., Oxford Univ. Press, New York, 1995.
\bibitem{M}
T. Machide, Extended double shuffle relations and the generating
function of triple zeta values of any fixed weight, \emph{Kyushu Math. J.}
{\bf 67} (2013), 281-307.
\bibitem{SC}
Z. Shen and T. Cai, Some formulas for multiple zeta values, 
\emph{J. Number Theory} {\bf 132} (2012), 314-323.
\end{thebibliography}
\end{document}